\documentclass[a4paper,12pt]{article}
\usepackage{amsthm}

\usepackage[ english]{babel}
\usepackage{amsmath}
\usepackage{amsfonts}
\usepackage{enumerate}
\usepackage{amssymb}
\usepackage[usenames]{color}
\usepackage{caption}

\newtheorem{theo}{Theorem}[section]

\newtheorem{prop}[theo]{Proposition}

\newtheorem{defi}[theo]{Definition}

\numberwithin{equation}{section}

\usepackage[a4paper]{geometry}
\usepackage{graphicx}
\usepackage{graphics}
\usepackage{epsfig}
\usepackage{amsmath}
\usepackage{amsfonts}
\usepackage{psfrag}

\begin{document}
\title{Parameter estimation for stochastic diffusion process with drift proportional to Weibull density function  }
\date{}
\maketitle

\vspace{ -1\baselineskip}

{\small
\begin{center}
{\sc H El otmany (version 0) } \\ 
helo@gmail.com\\

\end{center}
}
\smallskip

\begin{quote}
\footnotesize

{\bf Abstract.}
In the present paper we propose a new stochastic diffusion process with drift proportional to the Weibull density function defined as 
$$X_\varepsilon=x,\quad  dX_t=\left( \frac{\gamma}{t} (1-t^{\gamma+1})-t^{\gamma}\right) X_t dt + \sigma X_t d B_t,\quad t>0,$$ with parameters $\gamma>0$ and $\sigma >0$, where $B$ is a standard Brownian motion and $t=\varepsilon$ is a time near to zero. First we interested to probabilistic solution of this process as the explicit expression of this process. By using the maximum likelihood method and by considering a discrete sampling of the sample of the new process we estimate the parameters $\gamma$ and $\sigma$.

\textbf{Keyword} : Maximum LikeLihoode,  It\^o's formula, Weibul density, stochastic diffusion process, parameter estimation
\end{quote}

\section{Introduction}
In the present paper we propose a new stochastic Weibull process $X=\{ X_t,\, t>0\}$ given by the following linear stochastic differential equation 
$$ X_{\varepsilon}=x>0;\quad dX_t=\left( \frac{\gamma}{t} (1-t^{\gamma+1})-t^{\gamma}\right) X_t dt + \sigma X_t d B_t,\quad t>0 $$
where
$\mu(t, X_t;\gamma)= \frac{\gamma}{t} (1-t^{\gamma+1})X_t-t^{\gamma}X_t$ and $g( X_t;\sigma)=  \sigma X_t$.
We denote by $B$  a standard Brownian motion and $\gamma$ and $\sigma$ are an unknown parameters. An interesting problem is to estimate the parameters $\gamma$ and $\sigma$ are time independent reel parameters to be estimate when one observes the whole trajectory of $X$.

 The estimation for diffusion processes by discrete observation has been studied by several authors ( see for example \cite{PRAKASARAO1}, \cite{PRAKASARAO2}, \cite{Dacunha}, \cite{Florens}, \cite{Yoshida})and its references. Prakasa Rao \cite{PRAKASARAO2} treats this problem and shows that the least square estimator is asymptotically normal and efficient under the assumption $h\sqrt{N}\longrightarrow0$, the condition for ''rapidly increasing experimental design'' \cite{PRAKASARAO2}.
 
 The organization of our paper is as follows. Section 2 contains the presentation of the basic tools that we will need throughout the paper: basic properties of of standard Brownian motion and Itô's formula . The aim of section 3 is twofold. Firstly, we prove the close formula of SDE under the conditions $(1)-(3)$. Secondly, we invesigate the mean of $X_t$ an explicit solution of SDE and we prove that the drift is proportionnel to Weibul density. The section 4 is  devoted to estimate a parameters by using Maximum likelihood. In the last section we present  a numerical test.
  
 Given the general one-dimensional time-homogeneous SDE
\begin{equation}\label{SDE}
X_{\varepsilon}=x>0;\quad dX_t= \mu(X_t;\gamma) dt +  g( X_t;\sigma)d B_t,\quad t>0.
\end{equation}
where $\mu(t, X_t;\gamma)= \frac{\gamma}{t} (1-t^{\gamma+1})X_t-t^{\gamma}X_t$ and $g( X_t;\sigma)=  \sigma X_t$.
 The following conditions are assumed in this article :
 \begin{itemize}
 \item[(1)] There is a constant $L >0$ such that  
 $$ |\mu(t,x;\gamma)| + |g( x;\sigma)|\leq L (1 + |x|)$$
  \item[(2)] There is a constant $L >0$ such that  
  $$ |\mu(t,x;\gamma)- \mu(t,y;\gamma)| + |g( x;\sigma)-g( y;\sigma)|\leq L  |x-y|$$
  \item[(3)]For each $q>0$, $\sup_t\mathbb{E}(|X_t|^q)< \infty$.
 \end{itemize}

 \section{Preliminaries}
 
 In this section we start by recalling the definition of Brownian motion, which is a funda-
mental example of a stochastic process. The underlying probability space $\left(\Omega,\mathcal{F},\mathbb{P}\right)$ of Brownian motion can be constructed on the space $\Omega= \mathcal{C}_0(\mathbb{R}^+)$ of continuous real-valued functions on $\mathbb{R}^+$ started at $0$. For more complete presentation on the subject, see \cite{Lamberton}, \cite{Karatzas}.
 \begin{defi}
 The standard Brownian motion is a stochastic process $(B_t,\, t\geq 0)$ such that
 \begin{itemize}
    \item[(a)] $B_0$ almost surely;
   \item[(b)] With probability one $t\longrightarrow B_t$ is continuous.
    \item[(c)] For any finite sequence of times $t_0<t_1<\cdots<t_n$ the increments 
     $$ B_{t_1}-B_{t_0},\, B_{t_2}-B_{t_1},\,\cdots\, B_{t_n}-B_{t_{n-1}}$$ 
     are independent.
  \item[(d)] for any given times $0\leq s< t,\, B_t-B_s$ has the Gaussien distribution $\mathcal{N}(0,\,t-s)$ with mean zero and variance $t-s$.
 \end{itemize}
 We refer to Theorem 10.8 of \cite{Lamberton} and to Chapter of \cite{ Lamberton} for the proof of the existence of Brownian motion as a stochastic process  $(B_t,\, t\geq 0)$ satisfying the above properties $(a)-(d)$. In the sequel the filtration $(\mathcal{F}_t)_{t\geq 0}$ will be generated by the Brownian paths up to time $t$, in other words we write 
 $$ \mathcal{F}_t =\sigma(B_s:\, 0\leq s\leq t ),\quad t\geq 0.$$
 \end{defi}
 we give a basic properties of Brownian motion and extentions (\cite{DamienLamberton}):
 \begin{itemize}
 \item The crucial fact about Brownian motion, which we need is $(dB)^2=dt$;
 \item For every $0\leq s \leq t, \quad B_t -B_s$ is independent of $\{ B_u, \, u\leq s\}$ and has a $\mathcal{N}(0,t-s)$;
  \item Brownian motion $(B_t)$ is a process Markov property;
 \item $(-B_t)_{t\geq 0}$ is a Brownian motion;
 \item $X_t= X_0 +\mu t +\sigma B_t$ is a Brownian motion with drift with mean equal to $X_0 +\mu t$;
 \item $X_t= X_0 \exp(\mu t +\sigma B_t)$ is a Geometric Brownian motion with mean equal to $X_0 \exp(\mu t + \sigma^2/2)$.
 \end{itemize}
 We introduce the following two version of Itô's formula 
 \begin{theo}( It\^o's formula  v.1)
 Let $f\in \mathcal{C}^2(\mathbb{R})$. Then for $a<t$,
 \begin{equation}
 f(B_t)-f(B_a) =\int_a^t f'(B_s)\,ds +\frac{1}{2} \int_a^t f''(B_s)\,ds
 \end{equation}
 \end{theo}
 \begin{theo}( It\^o's formula  v.2)
 Let $=f(t,x)$ be a continuous in $[a,b]\times \mathbb{R}$ with $f_t,\, f_x,$ and 
 $f_{xx}$ continuos $(a,b)\times \mathbb{R}$. Then for $a<t<b$,
 \begin{equation}
 f(t,B_t)-f(a,B_a) =\int_a^t f_x(s,B_s)\,dB_s + \int_a^t\left(f_t(s,B_s)+\frac{1}{2} f_{xx}(s,B_s)\right)\,ds.
 \end{equation}
 \end{theo}
 \section{Closed formula and mean  of $X_t$}
By curiosity, we focus on the explicit formula for the SDE to find the mean and variance of $X_t$. By using the Itô's formula to $Y_t =\ln(X_t)$ we obtain 
\begin{equation}\label{ito log x 1}
dY_t= \left(\frac{\gamma}{t} (1-t^{\gamma+1})-t^{\gamma}-\frac{\sigma^2}{2}\right)dt+\sigma dB_t, \quad Y_\varepsilon =\ln x.
\end{equation}
By integrating between $\varepsilon$ and $t$ it follows
\begin{equation}
Y_t = \ln x + \gamma \ln \left(\frac{t}{\varepsilon}\right)
-(t^{\gamma+1}-\varepsilon^{\gamma+1})
 -\frac{\sigma^2}{2}(t-\varepsilon)   + \sigma (B_t -B_\varepsilon).
\end{equation}
 Then the explicit solution of SDE is given by 
 \begin{equation}
 X_t = x\left(\frac{t}{\varepsilon}\right)^\gamma \exp\left(-(t^{\gamma+1}-\varepsilon^{\gamma+1}) -\frac{\sigma^2}{2}(t-\varepsilon)+\sigma (B_t -B_\varepsilon)\right), \quad \forall t>0.
 \end{equation}
 We intersted now to mean of $X_t$. By using the conditional expectation or the Geomtric Brownian we prove that the trend of the process $X_t$ is given by
 \begin{equation}
 \mathbb{E}(X_t) = x\left(\frac{t}{\varepsilon}\right)^\gamma \exp\left(\varepsilon^{\gamma+1} -t^{\gamma+1}\right),
\end{equation}  
then it follows that the trend of $X_t$ is proportional to Weibul density.
 \section{Maximum likelihood estimators for the parameters of SDE }
A formal statement of the parameter estimation problem to be addressed is as follows. Given the general one-dimensional time-homogeneous SDE
\begin{equation}\label{SDE}
X_{\varepsilon}=x>0;\quad dX_t= \mu(X_t;\gamma) dt +  g( X_t;\sigma)d B_t,\quad t>0.
\end{equation}
where $\mu(t, X_t;\gamma)= \frac{\gamma}{t} (1-t^{\gamma+1})X_t-t^{\gamma}X_t$ and $g( X_t;\sigma)=  \sigma X_t$.

The task is to estimate the parameters $\theta =(\gamma,\sigma^2)$ of this SDE from a sample of $N +1$ observations $X_1,\,X_2,\cdots,\, X_N$  of the stochastic process at known times $t_1,\cdots, t_N$ . In the statement of equation \eqref{SDE}, $d B$ is the differential of the Brownian motion and the instantaneous drift $\mu(t,x;\gamma)$ and instantaneous diffusion $g( x;\sigma)$. Assuming that $\mathbb{P}(X_{t_1}= x)=p \neq 0$, for the sake of simplicity, let us assume that $p=1$.

 The ML estimate of $\theta$ is generated by minimising the negative log-likelihood function of the observed sample, namely
  \begin{equation}
LL(X_1,\,X_2,\cdots,\, X_n; \theta)= − \log f_1(X_1|\theta) - \sum\limits_{k=2}^N\log f(X_{k+1}| X_k;\theta)
 \end{equation}
 with respect to the parameters $\theta=(\gamma,\sigma^2)$. In this expression, $f_1(X_1|\theta)$ is the density of the initial state and $f(X_{k+1}| X_k;\theta) \equiv f((X_{k+1},t_{k+1}) | (X_k,t_k);\theta )$ is the value of the transitional PDF at $(X_{k+1},t_{k+1})$ for a process starting at $(X_k,t_k)$ and evolving to $(X_{k+1},t_{k+1})$ in accordance with equation \eqref{SDE}. Note that the Markovian property of equation \eqref{SDE} ensures that the transitional density of $X_{k+1}$ at time $t_{k+1}$ depends on $X_k$ alone.
 
ML estimation relies on the fact that the transitional PDF, $f (x, t)$, is the solution of the Fokker-Planck
equation
\begin{equation}
\frac{\partial f}{\partial t} = \frac{\partial }{\partial x} \left( \frac{1}{2} \frac{\partial\left( g( x;\sigma)f\right)}{\partial x}  - \mu(x;\gamma) f\right)
\end{equation}
satisfying a suitable initial condition and boundary conditions. Suppose, furthermore, that the state space of the problem is $[a,b]$ and the process starts at $x = X_k$ at time $t_k$. In the absence of measurement error, the initial condition is 
\begin{equation}
f (x, t_k) =\delta (x-X_k)
\end{equation}
where $\delta$ is the Dirac delta function, and the boundary conditions required to conserve unit density
within this interval are
\begin{equation*}
\lim\limits_{x\longrightarrow a^+} \left( \frac{1}{2} \frac{\partial\left( g( x;\sigma)f\right)}{\partial x}  - \mu(x;\gamma) f\right)=0,\quad \lim\limits_{x\longrightarrow b^-} \left( \frac{1}{2} \frac{\partial\left( g( x;\sigma)f\right)}{\partial x}  - \mu(x;\gamma) f\right)=0.
\end{equation*}
By using It\^o formula to $Y_t=\ln (X_t)$ we have the linear diffusion
\begin{equation}\label{ito log x}
dY_t= \left(\frac{\gamma}{t} (1-t^{\gamma+1})-t^{\gamma}-\frac{\sigma^2}{2}\right)dt+\sigma dB_t,
\end{equation}
By integrating between $t_k$ and $t_{k+1}$, the exact discret model correspond to \eqref{ito log x} is given by  
\begin{equation}
\ln (X_{k+1})= \ln (X_k) + \gamma \ln \left(\frac{t_{k+1}}{t_k}\right)
-(t_{k+1}^{\gamma+1}-t_k^{\gamma+1})
 -\frac{\sigma^2}{2}(t_{k+1}-t_k)+\sigma (B_{t_{k+1}}- B_{t_k}),
\end{equation}
From the above, the variance and esperance of $\ln (X_{k+1})$ is given by 
\begin{align*}
\mathbb{E}(\ln (X_{k+1})|X_k ) &= \ln (X_k)+ \gamma \ln \left(\frac{t_{k+1}}{t_k}\right)
-(t_{k+1}^{\gamma+1}-t_k^{\gamma+1})
 -\frac{\sigma^2}{2}(t_{k+1}-t_k),\\
  V(\ln (X_{k+1})| X_k)&=\sigma^2(t _{k+1}- t_k),
\end{align*}
then the transitional probability density function (PDF) for SDE  has the following closed form expression:
\begin{equation}
X_{k+1}| X_k \sim \mathcal{N}\left( \ln (X_k)+ \gamma \ln \left(\frac{t_{k+1}}{t_k}\right)
-(t_{k+1}^{\gamma+1}-t_k^{\gamma+1})
 -\frac{\sigma^2}{2}(t_{k+1}-t_k), \sigma^2(t _{k+1}- t_k)\right)
\end{equation}
 Now, the classical approach to the ML method requires the computation of the first-order partial derivatives of the log-likelihood function with respect to each of its parameters, equating them equal to zero and then solving the resulting system of equations. So, the first-order partial derivatives are obtained as follows:
 \begin{align}\label{gama derive}
 \frac{\partial LL}{\partial \gamma} &= \frac{1}{2\sigma^2} \sum\limits_{k=2}^N(\sigma^2+2 A_k(X_k,t_k))\left(\ln(\frac{t_k}{t_{k-1}})-(\gamma+1)
 (t_k^\gamma - t_{k-1}^\gamma)\right)=0,
 \\\label{sigma derive}
 \frac{\partial LL}{\partial \sigma^2}&=-\frac{n-1}{2\sigma^2}+\frac{1}{8\sigma^4} \sum\limits_{k=2}^N(\sigma^2+2 A_k(X_k,t_k))^2 
 -\frac{1}{4\sigma^2} \sum\limits_{k=2}^N(\sigma^2+2 A_k(X_k,t_k))=0
 \end{align}
 with $A_k(X_k,t_k)=\ln(\frac{X_k}{X_{k-1}})-\gamma \ln(\frac{t_k}{t_{k-1}})+ t_k^{\gamma+1} - t_{k-1}^{\gamma+1}.$\\
 From the equation \eqref{sigma derive} we obtain 
\begin{equation}
4(N-1) \hat{\sigma}^2 -(N-1)\hat{\sigma}^4 = 4 \sum\limits_{k=2}^N \hat A^2_k(X_k,t_k)
\end{equation}
By using the positivity of $\hat{\sigma}^2$ we have the following expression of estimator $\hat{\sigma}^2$
\begin{equation}
\hat{\sigma}^2 = \left(4-\frac{4}{N-1}  \sum\limits_{k=2}^N \hat A^2_k(X_k,t_k)\right)^{1/2}-2.
\end{equation}
Consequently, we prove the non-linear expression  of estimator $\hat \gamma$ by replacing $\hat{\sigma}^2$ in \eqref{gama derive} :
\begin{equation}
\sum\limits_{k=2}^N(\hat \sigma^2+2 \hat A_k(X_k,t_k))\left(\ln(\frac{t_k}{t_{k-1}})-(\hat \gamma +1)
 (t_k^{\hat \gamma} - t_{k-1}^{\hat \gamma} )\right)=0
\end{equation}
where $\hat A_k(X_k,t_k)=\ln(\frac{X_k}{X_{k-1}})-\hat \gamma \ln(\frac{t_k}{t_{k-1}})+ t_k^{\hat \gamma+1} - t_{k-1}^{\hat \gamma+1}.$

It is assumed that the observation from the realization consists of $X_{t_k}$, $t_1=\varepsilon,\, t_k= k h$ with $h>0,\, k=2,3,\cdots, N$.  We define the new estimator of $\sigma^2$ :
\begin{equation}
\hat{\sigma}^2 = \left(4-\frac{4}{N-1}  \sum\limits_{k=2}^N 
\left(\ln(\frac{X_k}{X_{k-1}})-\hat \gamma \ln(\frac{k}{k-1})+ h^{\hat \gamma+1}(k^{\hat \gamma+1} - (k-1)^{\hat \gamma+1}) \right)^2\right)^{1/2}-2.
\end{equation}
Assuming that $\gamma\cong \hat \gamma $. The following proposition is more or less well khnown.
\begin{prop}
If condition $(1)-(3)$ hold, then
$$\mathbb{E}(\hat{\sigma}^2 -\sigma^2) \leq C\left(\frac{1}{N-1}+ h)\right).$$
\end{prop}
\begin{proof}
By using integrating the following equation between $t_{k-1}= (k-1)h$ and $t_k=kh$:
\begin{equation}\label{ito log x}
dY_t= \left(\frac{\gamma}{t} (1-t^{\gamma+1})-t^{\gamma}-\frac{\sigma^2}{2}\right)dt+\sigma dB_t,
\end{equation}
By summing between $k=2$ and $N$ and using $\mathbb{E}(B_{t_k} -B_{t_{k-1}})^2 = h.$
Then we prove a result.
\end{proof}
\section{Numerical test}
We present numerical results in the software R for the following SDE:
 $$X_\varepsilon=10,\quad  dX_t=\left( \frac{\gamma}{t} (1-t^{\gamma+1})-t^{\gamma}\right) X_t dt + \sigma X_t d B_t,\quad t>0,$$
 with: and we generate sampled data $X_{t_k}$ with $\gamma =1$, $\sigma=0.2$ and time step $\Delta = 10^{-3}$ as following:
 \begin{itemize}
 \item mu $<$- expression( ((1./t)- 2*t)*x)
 \item g $<$- expression(0.2*x)
 \item Simulation $<$- SNSDE(drift=mu,diffus=g,N=1000,Dt=0.001,x0=10)
 \item MyData $<$- Simulation$X$
 \item mux $<$- expression( $gama*x./t-(gama+1)*t^{(gama+1)}*x$ )
 \item gx $<$- expression( $sigma*x$ ) 
 \item Model $<$- FDE(data=MyData,drift=mux,diffus=gx,start = list(gama=1,sigma=0.2),"euler")
 \item summary(Model)
 \end{itemize}
 Result :
 
 Pseudo maximum likelihood estimation
 \\
\hfill
\\
Method:  Euler

Call:

FSDE(data = MyData, drift = mux, diffus = gx, start = list(gamma=1, sigma=1),"euler")
The following table prove the estimate coefficients and standard error with $-2 \log L: 8794.337$ by unsig our methods. 
\begin{figure}[!h]
\begin{center}
\begin{tabular}{|l|c|r|}
  \hline
    & Estimate coefficient &Standard error\\
  \hline
   $\gamma$ & $0.9856006$ & $0.006651509$ \\
  $\sigma$ & $0.2310434$ & $0.005168557$ \\
  \hline
\end{tabular}
\caption{Estimate coefficients and  Standard  Error}
\end{center}
\end{figure}

By using the command in the R software " confint(Model,level$=0.9$)", we obtain the following table. It shows the confidence interval for the estimated variables $\gamma$ and $\sigma$.
\begin{figure}[!h]
\begin{center}
\begin{tabular}{|l|c|r|l|l|l|}
  \hline
   & $5 \%$ & $15\%$& $75\%$& $95 \%$\\
  \hline
  $\gamma$ & 0.9711487 &1.0000124 & 1.0003527&1.0012685 \\
  $\sigma$& 0.2225419&0.2238546 &0.2310215 &0.2395449\\
  \hline
\end{tabular}
\caption{ Confidence interval}
\end{center}
\end{figure}

\end{document}